\documentclass[11pt,reqno]{article}

\title{On the trace of branching random walks}
\author{Itai Benjamini \and Sebastian M\"uller}
\usepackage{amsmath,amsthm,amsfonts,amssymb}
\usepackage{t1enc}
\usepackage[latin1]{inputenc}
\newtheorem{thm}{Theorem}[section]
\newtheorem{conj}{Conjecture}[section]
\newtheorem{cor}[thm]{Corollary}
\newtheorem{lem}[thm]{Lemma}
\newtheorem{prop}[thm]{Proposition}

\theoremstyle{definition}

\theoremstyle{remark}
\newtheorem{rem}{Remark}[section]
\newtheorem{que}{Question}[section]

\def\orig{o}

\newcommand{\eps}{\varepsilon}

\newcommand{\En}{\mathcal{E}}

\newcommand{\G}{\mathcal{G}}

\newcommand{\Tr}{\mathtt{Tr}}

\newcommand{\Z}{\mathbb{Z}}
\newcommand{\N}{\mathbb{N}}
\newcommand{\T}{\mathbb{T}}
\newcommand{\1}{\textbf{1}}
\newcommand{\ro}{\textbf{r}}
\renewcommand{\P}{\mathbb{P}}

\newcommand{\PP}{\textbf{P}}

\newcommand{\E}{\mathbb{E}}

\newcommand{\GW}{\textbf{GW}}
\newcommand{\AGW}{\textbf{AGW}}
\newcommand{\UGW}{\textbf{UGW}}

\newcommand{\eval}[2][\right]{\relax
  \ifx#1\right\relax \left.\fi#2#1\rvert}

\begin{document}
\maketitle {\abstract We study branching random walks on  Cayley graphs. A first result is that the trace of a transient branching random walk on a Cayley graph is a.s.~transient for the simple random walk. In addition, it has a.s.~critical percolation probability less than one and exponential volume growth. The proofs rely on the fact that the trace induces an invariant percolation on the family tree of the branching random walk. Furthermore, we prove that the trace is a.s. strongly recurrent for any (non-trivial) branching random walk. This follows from the observation that the trace, after appropriate biasing of the root, defines  a unimodular measure. All  results are stated in the more general context of branching random walks on unimodular random graphs. }

\section{Introduction}
A branching random walk (BRW) is a \emph{cloud} of particles that move on an underlying graph $G$ in discrete time. The process starts with one particle in the root $\orig$ of the graph. At each time step each particle splits into offspring particles, which then move one step according to a random walk on $G$. Particles branch and move independently of the other particles and the history of the process. A first natural question is to ask whether the process eventually \emph{fills up} the whole graph, \emph{i.e.}, if every finite subset will eventually be full or free of particles. If the BRW visits the whole graph it is called recurrent and transient otherwise. In the transient case, the set of visited vertices and traversed edges defines a proper random subgraph of $G$ and its properties become of interest. This subgraph is called the trace of the BRW.

 One motivation of this note comes from the fact that the trace of a simple random walk (SRW) on a connected graph is a.s. recurrent with respect to the SRW, see \cite{BG-GL:07}. Recall that a random walk on $G$ is called recurrent if it returns a.s. an infinite number of times to the origin (or any finite subset of $G$) and transient otherwise. In \cite{BG-GL:07} it was conjectured that this phenomenon still holds true for BRWs.
First, we prove that the trace of a transient BRW on a unimodular random graph (URG) is in fact a.s.~transient for SRW, see Theorem \ref{thm:SRWonBRWtrace}, and then that it is a.s.~(strongly) recurrent for every BRW, see Corollary \ref{cor:tr_rec}. Our proofs rely on mappings of the family tree into the base graph of the BRW. In particular, we prove that there exists a unimodular random version of the trace.  The proof of Theorem \ref{thm:SRWonBRWtrace} uses the fact that the trace defines an invariant percolation on the family tree. This idea is also used to prove that the trace of a transient BRW on a URG has a.s. critical percolation probability less than~$1$ and exponential volume growth, see Theorem \ref{thm:critperc} and Proposition \ref{prop:expgrowth}. Besides this, we suggest a list of questions and conjectures about structural properties of the trace of BRWs.

\section{Preliminaries}

\subsection{Branching random walks}
We use the standard notation for a rooted graph $G=(V,E)$: $V$ is the set of vertices, $E$ is the set of edges, $deg(x)$ is the degree of $x$, we write $x\sim y$ if $(x,y)\in E$, and  denote $\orig$ for the root. We always assume the graph to be infinite, connected, and of bounded degree.

The branching random walk (BRW) starts with one particle in the root $\orig$ of the graph and is defined inductively: at each time step each particle splits into offspring particles, which then move one step according to a random walk on $G$. Particles branch and move independently of the other particles and the history of the process. We denote $(p_k)_{k\in \N}$ for the offspring distribution; $p_k$ is the probability that a particle splits into $k$ offspring. Let $m=\sum_k
p_k$ be the mean number of offspring. We will always assume that $p_0=0$ and $p_1<1$, \emph{i.e.}, that particles have at least one offspring, which guarantees the survival of the process, and that the process is not reduced to a non-branching random walk. The movement of the particles is described by the transition kernel $P=(p(x,y))_{x,y\in V}$ of a simple random walk (SRW) denoted by $(S_n)_{n\geq 1}$. Recall that SRW means that $p(x,y)=1/deg(x)$ if $x\sim y$ and $0$ otherwise, and that the connectedness of the graph assures the
irreducibility of the random walk. The probability distribution and expectation will be denoted by $\P_x$ and $\E_x$ for both the SRW and the BRW started in $x$. If not mentioned otherwise the processes always start in the root of the graph, and we write $\P$ and $\E$.

There is an alternative description of BRWs that uses the concept of tree-indexed random walks introduced in \cite{benjamini:94}. Let $\T$ be a rooted infinite tree. Denote by $v$ the vertices of $\T$ and let $|v|$ be the (graph) distance from $v$ to the root $\ro$.  The tree-indexed process $(S_v)_{v\in \T}$ is defined inductively such that
$$S_\ro=\orig~\mbox{and}~\P(S_v=x| S_{v^-}=y)=p(x,y), $$
 where $v^-$ is the unique predecessor of $v$, \emph{i.e.}, $v^{-}\sim v$ and $|v^-|=|v|-1$.
A tree-indexed random walk becomes a BRW if the underlying tree is a realization of a Galton--Watson process with offspring distribution $p=(p_k)_{k\geq 1}$. We call  $\T$ the family tree  and $G$ the base graph of the BRW.

An important class of  unimodular (random)  graphs, see Section \ref{subsec:URG}, are Cayley graphs.  In this case the BRW can be described as a labelled Galton--Watson tree. Let $G$ be a finitely generated group with group identity $\orig$ and write  the group operations multiplicatively. Let $q$ be the uniform probability measure on a finite symmetric generating set of $G$. The SRW on $G$ is the Markov chain with state space $G$ and transition probabilities $p(x,y)=q(x^{-1}y)$ for $x,y\in G$. Equivalently, the process (starting in $x$) can be described as
\[S_n=x X_1\cdots X_n,\quad n\geq1,\] 
where the $X_i$ are i.i.d. random variables with distribution $q$.
Now, label the edges of $\T$ with i.i.d. random variables $X_v$ with distribution $q$; the random variable $X_v$ is the label of the edge $(v^-,v)$. Define $S_v=\orig \cdot \prod_{i} X_{v_i}$ where $\langle v_0=\ro, v_1, \ldots, v_n=v\rangle$ is the unique geodesic from $\ro $ to $v$ at level $n$.

\subsection{Unimodular random graphs}\label{subsec:URG}
Unimodular random graphs (URG) or stochastic homogeneous graphs have several motivations and origins. In this note we concentrate on the probabilistic motivations since these give rise to the tools we are going to use. For more details on the probabilistic viewpoints we refer to \cite{AL:07}, and to \cite{KS:09} for an introduction to the ergodic and measure theoretical origins.

One motivation to consider unimodular random graphs is the use of a \emph{general} Mass-Transport Principle (MTP); this was established in \cite{BS:01} under the name of ``Intrinsic Mass-Transport Principle''. It was motivated by the fact the Mass-Transport Principle  is heavily used in the study of percolation and therefore lifts many results on unimodular graphs to a more general class of graphs. In \cite{AL:07} a probability measure on rooted graphs is called unimodular if this general form of the MTP holds. In \cite{Kai:98} a different language and a more general approach is used. In particular, unimodular measures on rooted graphs correspond to invariant measures of graphed equivalence relations, and unimodular graphs are called stochastic homogeneous graphs.

Let us now give a definition of a unimodular random graph (URG).
We write  $(G,\orig)$  for the graph $G$ with root $\orig$.   A rooted graph $(G,\orig)$ is isomorphic to $(G',\orig')$ if there is an isomorphism of $G$ onto $G'$ which takes $\orig$ to $\orig'$. We denote by $\G_*$ the space of isomorphism classes of rooted graphs and  write $[G,\orig]$ for the equivalence class that contains $(G,\orig)$. The space $\G_{*}$ is equipped with a metric that is induced by the following distance between two rooted graphs $(G,\orig)$ and $(G',\orig'):$ $d( (G,\orig), (G',\orig') )= \frac1{1+\alpha}.$
Here $\alpha$ is the  supremum of those $r>0$ such that there is some rooted isomorphism of the balls of radius $\lfloor r \rfloor$ (in graph distance) around the roots of $G$ and $G'$. 
In this metric  $\G_{*}$ is separable and complete. In the same way one defines the space $\G_{**}$ of isomorphism classes of graphs with an ordered pair of distinguished vertices.  A Borel probability measure $\mu$ on $\G_*$ is called unimodular if it obeys the Mass-Transport Principle: for all Borel function $f: \G_{**}\to[0,\infty]$, we have
\begin{equation}\label{eq:defMTP}
\int \sum_{x\in V} f(G,\orig,x)d\mu([G,\orig])= \int \sum_{x\in V} f(G,x,\orig) d\mu([G,\orig]).
\end{equation}
Observe that this definition can be extended to labelled graphs or networks. A network is a graph $G=(V, E)$ together with a complete metric space $M$ and maps from $E$ and $V$.
While the definition of the above equivalence classes for networks is straightforward, one has to adapt the metric between two networks as follows: the $\alpha$ is chosen as the supremum of those $r>0$ such that there is some rooted isomorphism of the balls of radius $\lfloor r \rfloor$ around the roots of $G$ and $G'$ and such that each pair of corresponding labels has distance at most $1/r$.

Another way to look at unimodular measures uses random walks on rooted graphs. Instead of  considering a random walk on a graph and observing its position  on the graph,  we keep track of the environment seen from the point of view of the particle. The state space is then the space of rooted graphs, where the position of the random walk corresponds to the root. Furthermore, there is a one-to-one correspondence between stationary measures of the environment seen from the particle and unimodular measures on rooted graphs: the density of the stationary measure with respect to the unimodular measure is  the vertex degree function, see \cite{Kai:98}. Observe that in \cite{Kai:98} this connection is given in terms of invariant measures for treed-equivalence relations.

To illustrate this connection let us consider an example that is important for us: the Galton--Watson measure. Let $p=\{p_k\}_{k\in\N}$ be a probability distribution on the integers.  The Galton--Watson tree is defined inductively: start with one vertex, the root of the tree. Then, the number of offspring of each particle (vertex) is distributed according to $p$. Edges are between vertices and their offspring. We denote by \GW~the corresponding measure on the space of rooted trees. We always assume that $p_0=0$ which implies that the tree is infinite a.s. and has no leaves. In this
construction the root clearly plays a special role. In \cite{LPP:95} the augmented Galton--Watson measure (\AGW) was introduced where the root has $k+1$ offspring with probability $p_k$ and they showed that \AGW~is the stationary measure for the environment seen from the point of view of the SRW. In the unimodular Galton--Watson measure (\UGW) the root has a degree biased distribution: the probability that the root has degree $k+1$ is $\frac1c \cdot\frac{p_{k}}{k+1}$ with $c= \sum_{i} (p_{i}/(i+1))$.  In cases where we use the \UGW~measure instead of the standard \GW~measure to define the family tree of the BRW we denote the BRW by UBRW.

\subsection{Basic results}
One first question to ask is whether the trace is a proper random subgraph of the base graph. This is equivalent to the question of recurrence of the process. Recall that a (non-branching) random walk is called recurrent if it returns infinitely many times to its starting point and transient if it eventually leaves every finite set. This definition can be generalized to BRWs modulo the following observation. Let $\alpha(x)$ be the probability that a BRW started in $x$ visits $x$ an infinite number of times. 
Now, irreducibility of the underlying SRW guarantees that the following terms are well-defined:  a BRW is called strongly recurrent if $\alpha(x)=1$ for all (some) $x\in G,$ weakly recurrent if $0<\alpha(x)<1$ for all (some) $x\in G,$ and transient if $\alpha(x)=0$ for all (some) $x\in G$. We say the BRW is recurrent if it is not transient.  (Notice that strong recurrence is equivalent to guaranteed return, \emph{i.e.}, the process returns to the starting position almost surely.) 
While in general a BRW may be weakly recurrent, in  \emph{homogeneous} cases, see  \cite{mueller08} and Theorem \ref{thm:rec_unimod} below, a recurrent BRW is always strongly recurrent. We refer to \cite{mueller08} where more references and details about the different types of BRWs can be found.

It turns out that recurrence and transience depend on \emph{local} properties of the graph and can be classified using the spectral radius of the random walk: $$\rho(P)=\limsup_{n\to\infty} (p^{(n)}(x,y))^{1/n}.$$ Note that, due to the irreducibility, $\limsup (p^{(n)}(x,y))^{1/n}$ does not depend on~$x$ and~$y$. We write $\rho(G)$ for the spectral radius of the SRW on $G$.

\begin{thm}\label{thm:rec_trans}\cite{gantert:04}
The BRW with underlying irreducible Markov chain $P$ is recurrent if $m>1/\rho(P)$ and transient otherwise.
\end{thm}

\begin{rem}
Recall the well-known amenability criterion of Kesten: a finitely generated group is amenable if and only if the SRW on its Cayley graph  has spectral radius $1$, e.g. see Section 12.A. in  \cite{woess}.  An immediate consequence of the latter result and Theorem \ref{thm:rec_trans}  is that a finitely generated group is amenable if and only if any BRW with $m>1$ is recurrent on its Cayley graph.
\end{rem}
In \emph{homogeneous} cases, as Cayley graphs, quasi-transitive graphs, i.i.d. random environment, a $0-1$-law for $\alpha$ is established in \cite{mueller08}. This fact generalizes to BRWs on unimodular random graphs.
\begin{thm}\label{thm:rec_unimod} 
Let $\mu$ be a unimodular measure. Then for $\mu$-a.a. $G$ the BRW is strongly recurrent if $m>1/\rho(G)$ and transient otherwise.
\end{thm}
\begin{proof}
The fact that $m\leq 1/\rho(G)$ implies transience is just one part of Theorem \ref{thm:rec_trans}. The other part of this Theorem ensures that $m> 1/\rho(G)$ implies recurrence and hence that $\alpha(x)>0$ for all $x$. Now, as in \cite{mueller08}, the idea of the proof is to find a (random) sequence $(y_{j})_{j\in\N}$ of vertices that are visited by the BRW and satisfy $\alpha(y_{j})>c$ for some $c>0$. This implies that at least one of the $y_{j}$ will be visited infinitely many times. Eventually, the irreducibility of the SRW implies that every vertex is visited infinitely many times, hence $\alpha(x)=1$ for all $x$.

Let us fix one geodesic $\langle \ro, v_1, v_2,\ldots\rangle$ in the family tree. The values of $S_v$ along the geodesic correspond to the values of a (non-branching) random walk. Denote by $(G_n, \orig_n)$ the environment process with $(G_0,\orig_0)=(G,\orig)$ of this random walk. Observe that  unimodularity of the base graph implies that the environment process is stationary and recall that $\orig_{n}$ corresponds to the position of the SRW at time $n$. In particular, the $\alpha_n=\alpha(\orig_n)$ form a stationary and strictly positive sequence. Hence there exists some constant $c>0$ such that there exists some (random) subsequence $y_{j}=x_{n_{j}}$  of $\orig_{n}$ with $\alpha(y_{j})>c$.   At each time $n_{j}$, $k-1$ new independent BRWs are started from the geodesic (in $y_{j}$) with probability $p_k$. Each of those has probability $\alpha(y_{j})>0$ that infinitely many particles visit $y_{j}$, hence there exists a random element $y^{*}$ of $(y_{j})_{j\in\N}$ such that $y^{*}$ is visited infinitely many times. 

\end{proof}
\begin{rem}
Note that one also could use a \emph{seed}-argument introduced in \cite{comets:05} together with the MTP. A seed is a finite subset of $G$ such that the process restricted to this set is a supercritical Galton--Watson process. By MTP it follows that infinitely many seeds are visited.
\end{rem}

\begin{rem}
While $\rho(G)$ in Theorem \ref{thm:rec_unimod}  may in general be random, it is deterministic if the measure $\mu$ is extremal.
\end{rem}

\section{Properties of the trace}
Before asking whether the trace, denoted by $\Tr$, is recurrent for BRWs it is appropriate to ask if it is transient for the  SRW. For BRWs on homogeneous trees it was shown in \cite{hueter:00} that the trace is a.s. transient for the SRW. We extend this result to BRWs on unimodular random graphs.

\begin{thm}\label{thm:SRWonBRWtrace}
The trace of a transient BRW on a unimodular random graph is a.s. transient for the SRW.
\end{thm}
\begin{proof} 
The proof relies on the interpretation of the BRW as a tree-indexed random walk. The main idea is roughly the following. Since the SRW on the Galton--Watson tree is transient there exists a unit flow with finite energy from $\ro$ to infinity. We map this flow into the base graph in order to obtain a unit flow. The crux is then to show that the flow in the base graph has finite energy. In order to control the latter energy we consider appropriate subgraphs of the family tree.

Let us just recall the basic definitions and notations; we refer to \cite{lyons:book} for more details. Directed edges are denoted by $e=\langle e^-, e^+\rangle$. A flow $\theta$ is an antisymmetric real valued function on the edge set. The energy of a flow is defined as
$$\En(\theta)=\|\theta\|_r^2= \frac12 \sum_e r(e) \theta(e)^2,$$ where $r(\cdot)$ denotes the resistances of the edges. A flow $\theta$ is a unit flow from $a\in V$ to infinity if for all $x\in V:~\sum_{e^-=x} \theta(e)=\1_{\{a\}}(x)$. There is the well-known criterion for recurrence and transience for electrical networks due to \cite{TL:83}: the random walk on a countable infinite connected network $G$ is transient iff there exists a unit flow with finite energy on $G$ from some (every) vertex to infinity.

We can use the tree-indexed random walk to define a random mapping of the family tree $\T$ to $G$: the edge $\langle v^{-}, v \rangle$ in $\T$ is mapped to the edge $\langle S_{v^-}, S_v\rangle$ in $G$.  The above mapping enables us to define a percolation on the tree $\T$. 
Let $N>0$ and define $\T_N$ as the induced subgraph that consists of all edges $$\{\langle v^-, v\rangle:~ |T_{\langle S_{v^-}, S_v\rangle}|\leq N \},$$
where   
$$T_{\langle x, y\rangle}=\{v\in \T:~ \langle S_{v^-}, S_{v}\rangle=\langle x, y\rangle\}.$$

Let us first assume that the family tree $\T$ is a homogeneous tree. Since $G$ is a unimodular random graph  $\T_N$ defines an invariant percolation of the family tree. Theorem 1.6 in \cite{Hag:97} guarantees the existence of infinite clusters in this percolation process with sufficiently high marginal ($N$ sufficiently large). Furthermore, by Theorem 1.3 and Theorem 1.5 in \cite{Hag:97}, the branching number  of an infinite component is strictly larger than $1$. Note that in order to apply Theorem 1.3 in \cite{Hag:97} the percolation  $\T_{N}$ has to  satisfy the finite energy condition. This can be easily achieved by replacing $\T$ by a Bernoulli($p$)-percolation of $\T$. Hereby we have to choose $p$ sufficiently large to ensure that $\T_{N}$ has sufficiently large marginal.

The fact that the branching number of $\T_{N}$ is strictly larger than $1$ implies the transience of the SRW on infinite clusters of $\T_N$ for $N$ sufficiently large. We can assume that the root $\ro$ is part of an infinite cluster and let $\theta_N$ be a unit flow of finite energy from $\ro$ to infinity in $\T_N$.  A flow on $\T_{N}$ induces a flow $\theta_{N,G}$ on $G$: let $\langle x, y\rangle$ be an edge in $G$ and define
$$\theta_{N,G}(\langle x,y\rangle)=\sum_{v\in T_{N,\langle x,y\rangle}} \theta(\langle v^{-},v\rangle),$$ where
 $$T_{N,\langle x, y\rangle}=\{v\in \T_{N}:~ \langle S_{v^-}, S_{v}\rangle=\langle x, y\rangle\}.$$

Due to the construction of $\T_N$ the induced flow $\theta_{N,G}$ in~$G$ is a unit flow of  finite energy. Hence the subgraph of the trace that consists of all edges that were visited less then $N$ times and contains the origin is transient for the SRW. Since the existence of a transient subgraph implies transience of the whole graph, see e.g. Corollary 2.15 in  \cite{woess}, the trace of the BRW is transient too. 

For the general family tree we use Theorem~2 in \cite{Dekking:91}: there exists some constant $K$ such that the family tree contains a full binary tree whose edges are stretched to a path of length $K$. Now, we can argue as above by considering the trace of the random walk  indexed by the stretched binary tree.
\end{proof}

We want to highlight the usefulness of the concept that underlies the proof of Theorem
\ref{thm:SRWonBRWtrace} and give several applications: Theorem \ref{thm:critperc}, Proposition \ref{prop:expgrowth}, and Lemmata~\ref{lem:ends}   and \ref{lem:cutpoints}. 

Let us consider Bernoulli$(p)$ percolation on a locally finite graph $G$; for fixed $p\in[0,1]$, each edge is kept with probability $p$ and removed otherwise  independently of the other edges. The random subgraph that remains after percolation is denoted by $\omega$. Denote $\PP_p$ for the corresponding probability measure and define the critical probability
$$p_c(G)=\sup\{p:~\PP_p(\exists~\mbox{infinite component of}~\omega)=0\}.$$
\begin{thm}\label{thm:critperc}
Let $\Tr$ be the trace of a transient BRW on a URG. Then, $$p_c(\Tr)<1~a.s.$$
\end{thm}
\begin{proof}
We have to prove that for some $p<1$ the trace contains an infinite connected cluster. Consider the tree $\T_N$ defined in the proof of Theorem \ref{thm:SRWonBRWtrace} and recall that $\T_N$ defines an invariant percolation. Let us define yet another invariant percolation $\T_{N,p}$ via the Bernoulli percolation on the trace: $\T_{N,p}$ consists of those vertices of $\T_N$ that are mapped to an edge of $\Tr$ that remains after percolation of the trace. Observe, that erasing an edge of $\Tr$ corresponds to erasing at most $N$ edges of $\T_N.$ Using Theorem 1.6 in \cite{Hag:97}  we can choose first $N$ and then $p<1$
sufficiently large such that the marginal of the percolation defined by $\T_{N,p}$ is sufficiently large  to guarantee the existence of an infinite cluster. We conclude by observing that an infinite cluster of $\T_{N,p}$ is mapped onto a infinite subgraph of the percolated trace.
\end{proof}

Let $B^{(n)}$ be the ball around the origin of radius $n$ and define $\Tr^{(n)}=\Tr\cap B^{(n)}$.

\begin{prop}\label{prop:expgrowth}
The trace of a transient BRW on a URG has a.s.~exponential volume growth, \emph{i.e.}, $\exists
c>0$ and $r>1$ such that $|\Tr^{(n)}|\geq c r^n$ for all $n\geq 1$.
\end{prop}

\begin{proof}
Let us assume that the family tree is a binary tree; the case of a general family tree is treated as in the proof of Theorem \ref{thm:SRWonBRWtrace}. There exists some $N$ such that the tree $\T_{N}$ has branching number greater than $1$ and hence grows exponentially fast. This means that  there exist some constants $b>0$ and $r>1$ such that $|\T_{N}\cap B^{(n)}|\geq b r^{n}$.
Now, by mapping $\T_{N}$ to the base graph we obtain a subgraph of the trace that we denote by  $\Tr_{N}$. Observe that each edge $\langle v^{-}, v\rangle$ is mapped to the edge $\langle S_{v^{-}}, S_{v}\rangle$ with $|v|\geq |S_{v}|$, where $|S_{v}|$ is the graph distance in the base graph  from $S_{v}$ to the root $\orig$. Eventually, $|\Tr_{N}\cap B^{(n)}|\geq |\T_{N}\cap B^{(n)}|/N\geq (b/N) r^{n}$. We conclude by observing that  $\Tr^{(n)}\supseteq\Tr_{N}^{(n)}$.
\end{proof}

\begin{lem}\label{lem:cutpoints}
The trace of a transient BRW on a URG has a.s. only finitely many cutpoints, \emph{i.e.}, points that separate the root from infinity.
\end{lem}

\begin{proof}
Let us put in the setting of the proof of Proposition \ref{prop:expgrowth} and  choose $N$ such that  $\T_{N}$ has a branching number greater than $1$.
Recall a definition of the branching number of a tree that uses cutsets of trees. A cutset $\Pi$ of a tree is a subset of vertices that separates the root $\ro$ from infinity, \emph{i.e.}, every infinite geodesic path starting from $\ro$ goes through a vertex of $\Pi$. Now the branching number of a tree $T$ is defined as
\begin{equation*}
br(T):=\sup\left\{\lambda\geq 1:~\inf_{\Pi \mbox{  cutset}} \sum_{v\in \Pi}\lambda^{-|v|}>0\right\}.
\end{equation*} Since $br(\T_{N})>1$ there exists some $\lambda>1$ and $\eps>0$ such that  $\sum_{v\in \Pi}\lambda^{-|v|}>\eps$ for all cutsets $\Pi$. Define the height of a cutset $\Pi$ as $h(\Pi)=\min \{|v|:~v\in \Pi\}$. Now, $\sum_{v\in \Pi} \lambda^{-|v|}\leq |\Pi| \lambda^{-h(\Pi)}$ and hence $|\Pi|\geq \eps \lambda^{h(\Pi)}$.

 Assume that $x\in \Tr_{N}$ is a cutpoint and denote $n$ for its distance to the origin in $\Tr_{N}$. Observe that all $v\in \T_{N}$ are mapped to a vertex $S_{v}$ with $|S_{v}|\leq |v|$. Hence, there exists some cutset of $\T_{N}$ such that $h(\Pi)\geq n$ and $|\Pi|\leq N$. Together with $|\Pi|\geq \eps \lambda^{h(\Pi)}$ for all cutsets $\Pi$, this implies that the to $\T_{N}$ corresponding trace $\Tr_{N}$ has a.s. only a finite number of cutpoints. We can conclude by observing that  no vertex of $\Tr\setminus \Tr_{N}$ can be a cutpoint of $\Tr$; the root $\orig$ is connected to $\infty$ through $\Tr_{N}$.
\end{proof}

The next result concerns ends of the trace of BRWs. Recall the basic notations. 
A ray is an infinite path $(x_{0},x_{1},\ldots)$ of distinct vertices. A set $F$ separates sets of vertices $A$ and $B$ if any path from any vertex in $A$ to any vertex in $B$ contains an element of $F$. Two rays are equivalent if they cannot be separated by a finite set of vertices. The equivalence  classes of rays are called ends. 

There are several approaches to define the thickness of ends; we follow the one that uses cuts, e.g. see \cite{KrMo}. Denote $\partial C$ the (interior) boundary of a subset of vertices, \emph{i.e.}, $\partial C:=\{x\in C:~\exists  y\notin C  \mbox{ with }y\sim x\}$. If $\partial C$ is finite we call the set $C$ a cut.
 A ray $R$ lies in a set of vertices $C$ if all but finitely many vertices of $R$ are elements of $C$. An end $\omega$ lies in a set of vertices $C$ if all rays in $\omega$ lie in $C$. We say that a set of vertices $F$ separates two ends $\omega_{1}$ and $\omega_{2}$ if it separates any ray in $\omega_{1}$ from any ray in $\omega_{2}$. If a ray $R$ lies in a cut then the end which $R$ belongs to lies in $C$. Furthermore, each pair of distinct ends is separated by a finite set of vertices. 
The thickness of an end $\omega$ is the smallest  $t(\omega)\in \N\cup\{\infty\}$ such that there is a descending sequence of cuts $(C_{n})_{n\in\N}$ which contain $\omega$, such that $|\partial C_{n}|\leq t(\omega)$, and $\bigcap_{n} C_{n}=\emptyset$. An end is called thin if $t(\omega)<\infty$ and thick otherwise.

\begin{lem}\label{lem:ends}
The trace of a BRW on a URG with infinitely  many thin ends and no thick ends has a.s.~infinitely many ends.
\end{lem}

\begin{proof}
We proof the claim by contradiction. Assume that the trace has finitely many ends, say $\omega_{1},\ldots, \omega_{k}$. Then for each of these ends there exists a sequence of cuts $C_{n}^{(i)}$,  $1\leq i \leq k$. Now, $C_{n}=\bigcup_{i} C_{n}^{(i)}$ defines a cut that separates $0$ from infinity. Note that $|C_{n}|\leq \sum_{i} t(\omega_{i})$ for all $n$. Arguing as in the proof Lemma \ref{lem:cutpoints} leads to a contradiction with the fact that the branching number of $\T_{N}$ is greater than $1$ for some $N$.
\end{proof}

The next observation is about the spectral radius of the trace of a transient BRW (say with mean offspring $m$) on a general graph. Together with Theorem \ref{thm:rec_trans} it implies that for every $m'>1$ with positive probability the trace is recurrent for a BRW with mean offspring $m'$.
\begin{lem}\label{lem:weakrec}
Let $G=(V,E)$ be a locally finite graph and $\Tr$ the trace of a transient BRW on $G$. Denote $\rho(\Tr)$ the
spectral radius of the SRW on $\Tr$. Then, for all $\eps
>0$
\begin{equation}
\P(\rho(\Tr)\geq 1-\eps)>0.
\end{equation}
\end{lem}
\begin{proof}
Let  $F$ be a subset of $V$ and let $P_F$ be the substochastic matrix over $F$ defined as $p_F(x,y)=p(x,y)$ if $x,y\in F$ and $0$ otherwise. We use the \emph{finite approximation} property of the spectral radius: $\rho(P)=\sup_{|F|<\infty} \rho(P_F)$, where $\rho(P_{F})$ is the largest eigenvalue of $P_{F}$. Denote $Q$ for the transition matrix of the SRW on $\Z$ and let $L_k$ be the line segment of length $k$. It is well known that $\rho(Q)=1$; hence for each $m>1$ there exists some $k$ such that $\rho(Q_{L_k})>1/m.$ 

We say that the trace contains a line segment $L_{k}$ if there exists a sequence of vertices $x_{0}, x_{1},\ldots, x_{k}$ such that $x_{i}\sim x_{i+1}$ for $i=0,\ldots,k-1$ and $deg(x_{i})=2$ for all $i=1,\ldots,k-1$. It remains to prove that for each $k$ the trace contains line segments of length $k$ with positive probability.

Since the BRW is transient there exists some vertex $x_{k}$ with $|x_{k}|=k$  such that the BRW started in $x_{k}$ does not hit the ball $B^{(k-1)}$ around $\orig$. 
Let $(\orig, x_{1}, x_{2}, \ldots, x_{k})$ be a path from $\orig$ to $x_{k}$ of length $k$ and $i_{min}$ the smallest integer $i$ such that $p_{i}>0$. Now, with positive probability the following can happen: the BRW starts in $\orig$  with one particle that produces $i_{min}$ offspring. These offspring particles all jump to $x_{1}$ and each of them  produces $i_{min}$ offspring that all jump to $x_{2}$. We proceed in this way such that at time $k$ all existing  $i_{min}^{k}$ particles are at $x_{k}$ and no vertex outside the set $\{\orig, x_{1}, x_{2}, \ldots, x_{k}\}$ was visited. The rest of the process behaves like $i_{min}^{k}$ independent BRWs started at $x_{k}$. Hence, with positive probability no particle ever return to $B^{(k-1)}$ and the trace contains a line segment of length $k$.
\end{proof}

\begin{thm}\label{thm:tr=unimod}
The trace of a UBRW on a URG is a.s. a unimodular random graph.
\end{thm}
\begin{proof}
As in the proof of Theorem \ref{thm:SRWonBRWtrace} we consider the mapping of the family tree to the graph. Now, every set of edges in the tree that is mapped to the same edge in the base graph gets the same label. In other words, all elements of $\T_{\langle x,y  \rangle}$ are labelled the same.  This labelling is invariant under re-rooting and thus the labelled tree is a unimodular random  labelled tree. This shows in particular that the trace does not depend on the choice of the root.

Another way to see this is to check that the generalized MTP, Equation  \eqref{eq:defMTP}, holds. First, consider BRWs on Cayley graphs. We denote by $\T$ the labelled \UGW-tree with corresponding measure $\nu$; the labels are taken from the set of generators according to the definition of the UBRW as a tree-indexed random walk. Define a labelled version of the trace: an edge of the trace is labelled by the number of traversal of the BRW. We will prove that this labelled trace is a random unimodular network. We also use $\Tr$ for the notation of the labelled version and write $\mu$ for its probability measure.

We write $\{\T \rightsquigarrow \Tr,\ro\}$ or just $\{\T \rightsquigarrow \Tr\}$  for the set of rooted trees that generate  the rooted trace $\Tr$. The root of $\T$ is denoted by $\ro$ and the one of $\Tr$ by $\orig$. For any given $\Tr,$ $\tilde\T\in\{ \T \rightsquigarrow \Tr\}$, and $x\in \Tr$ let $E(x)$ be the set of vertices of $\tilde\T$ that map to $x$ and define
$$g(\tilde\T,\ro,v)=f(\Tr,\orig,x) |E(x)|^{-1}, \mbox{ if } v\in E(x), $$ 
and $g(\tilde\T,\ro,v)=0$ otherwise. Since $|E(x)|$ is constant and finite on $\{\T \rightsquigarrow \Tr\}$ for a given $x$ the latter is well defined. Now,
\begin{eqnarray*} & &\int_\Tr \sum_{x\in \Tr} f(\Tr,\orig,x) d\mu[\Tr,\orig]= \int_{\Tr}  \sum_{x\in
\Tr} f(\Tr,\orig,x) d\nu[\T\rightsquigarrow \Tr,\ro]\\
& & = \int_{\Tr}\int_{\tilde\T\in \{\T \rightsquigarrow \Tr\}} \sum_{x\in \Tr}  \sum_{v\in E(x)}  g(\tilde\T,\ro, v) \frac{d\nu[\tilde\T\in\{ \T\rightsquigarrow \Tr,\ro\}]}{\nu[\T \rightsquigarrow \Tr,\ro]} d\nu[\T\rightsquigarrow \Tr,\ro]\\
& &  =\int_{\T} \sum_{v\in \T} g(\T,\ro, v) d\nu[\T,\ro]\\
\end{eqnarray*}
and unimodularity of $\mu$ follows by unimodularity of $\nu$. The proof for the more general case of unimodular random graphs is in the same spirit. Denote by $BRW_{G,\T}$ the measure for the tree-indexed process with family tree $\T$ and base graph $G$. Let $\Tr$ be a labelled trace and define for $(G,\orig)$, $(\T,\ro)$ and $x\in G,~v\in\T$
$$ g(G,\T,\orig,x,\ro,v)=\int f(\Tr,\orig,x) |E(x)|^{-2} dBRW_{G,\T},\mbox{ if } v\in E(x), $$ and $0$ otherwise.
We can conclude as above using the unimodularity and independence of the measures of the family tree and the base graph.
\end{proof}

Now, we just have to combine Lemma \ref{lem:weakrec}  and Theorem \ref{thm:rec_unimod}:
\begin{cor}\label{cor:tr_rec}
Consider a transient BRW on a URG. Then, the spectral radius of the trace  is a.s.~$1$.  Furthermore, every BRW on a.e.~trace is strongly recurrent. 
\end{cor}
\begin{proof} As in the proof  of Lemma \ref{lem:weakrec} let $\eps>0$ and $k$ such that the spectral radius of the line segment $L_{k}$ is greater than $1-\eps$.  The proof of Lemma \ref{lem:weakrec} shows that the root of the trace belongs to a line segment $L_{k}$ with positive probability. Since the trace is a unimodular random graph it follows that it contains such a line segment a.s. and hence that $\rho(\Tr)\geq 1-\eps$ a.s.. The fact that  any BRW is strongly recurrent on the trace is now a consequence of Theorem \ref{thm:rec_unimod} and Theorem \ref{thm:tr=unimod}.
\end{proof}

\begin{rem}
In Theorem \ref{thm:tr=unimod} we proved that the labelled trace is a random unimodular network. Recall that the labels have been the number of times an edge was visited. Denote $N(x,y)$ for the label of the edge $\langle x,y\rangle$ and define a random walk where the probability to take  $\langle x, y\rangle$ is proportional to $N(x,y)$, \emph{i.e.}, $p_N(x,y)=N(x,y) (\sum_{z\sim x} N(x,z))^{-1}$. The above arguments apply to this model and we obtain that for~a.a.~labelled traces of transient BRWs (on URG) the BRW with transition kernel $P_N$ and mean offspring $m>1$ is strongly recurrent.
\end{rem}

\section{Discussion}
We want to use the opportunity to discuss briefly some questions and conjectures that are related to our results
above and may stimulate further research on BRW on graphs.

\subsection{Unimodular random graphs and Cayley graphs}

In~\cite{BLS} the speed of SRWs on Bernoulli percolation clusters on non-amenable Cayley graphs was studied. The trace of BRWs on non-amenable graphs share some similarities with percolation clusters.  Even though the trace turns out to be an amenable graph we believe due to exponential growth that the SRW on the trace has positive speed:

\begin{conj}\label{conj:speed}
The SRW on a.e.~trace of a transient BRW on a URG has  positive speed.
\end{conj}

Note that the speed of the SRWs on traces of transient BRWs on
URG  exists, and is deterministic if the unimodular measure is extremal. This follows from the fact that the environment seen from the point of view of the particle is stationary, and ergodic if the unimodular measure is extremal, e.g. see \cite{AL:07}. In fact, Lemma \ref{lem:ends} together with Proposition 4.9 and Theorem 6.2 in \cite{AL:07} implies that the SRW has positive speed on the trace of a transient BRW on a unimodular random tree with infinitely many ends. 
The question of positive speed is connected to non-amenability of the unimodular measure: Theorem 8.15 in \cite{AL:07} states that for unimodular and non-amenable measures $\mu$ concentrated on graphs with bounded degrees the speed of SRW is $\mu$-a.s. positive. The measure of the trace of a transient BRW on a unimodular random tree with infinitely many ends is non-amenable, see Lemma  \ref{lem:ends} and Corollary 8.10 in \cite{AL:07}. We conjecture this to hold more generally:

\begin{conj}\label{conj:non-amen}
Let $\mu$ be the measure of the trace of a transient BRW on a URG. Then, $\mu$ is non-amenable.
\end{conj}

On Cayley graphs positive speed is equivalent to admitting non-constant bounded harmonic functions, see~\cite{KV}. In~\cite{BC} this equivalence was extended to URG. Thus if Conjecture~\ref{conj:speed} is true, we expect non-constant bounded harmonic functions on the trace. It is of interest to study the Poisson and Martin boundary. In
particular, one might ask if the result of Lemma \ref{lem:ends} holds in general:

\begin{que}\label{que:ends}
Does the trace of a transient BRW on a Cayley graph have a.s.~infinitely many topological ends?
\end{que}

Theorem \ref{thm:critperc} states that the critical percolation probability is strictly less than $1$.  Observe that Conjecture \ref{conj:non-amen} would imply (under a first moment condition) that a.s.~there is no infinite cluster in Bernoulli($p_{c}$) percolation of the trace, see Theorem 8.11 in \cite{AL:07}.  Due to exponential growth and unimodularity of the trace one might expect mean-field criticality for percolation on the trace.

\begin{que}  Does the triangle condition hold for percolation on the trace of a transient BRW on URG? Is there mean-field criticality for percolation on the trace? We refer to and \cite{Kozma} and \cite{Schon:02}  for further details on these questions.
\end{que}

Hueter and Lalley \cite{hueter:00} studied BRWs on homogeneous trees. Observe that in their setting and notation weak survival is equivalent to transience in our language. In the transient regime the BRW eventually vacates every finite subset and the particle trails converge to the geometric boundary $\Omega$ of the tree. Let $\Lambda$, called \emph{limit set} of the BRW, be the random subset of the boundary that consists of all ends that are visited infinitely often by the process. In \cite{hueter:00} it is shown that the limit set has Hausdorff dimension no larger than one half the Hausdorff dimension of the entire boundary $\Omega$. 

Recall that a vertex $x$ is a furcation point, if removing $x$ would split the trace into at least $3$ infinite clusters. An application of the MTP shows that the number of furcation points is a.s. $0$ or $\infty$. Furthermore, one might conjecture that the trace of a transient BRW  has infinitely many ends, compare with Question \ref{que:ends}. Hence, it would be interesting to know how the Hausdorff dimensions of the limit sets compare in general to the one of the full boundary.

We suspect that the Hausdorff dimension of the limit set (observe that for $m>1/\rho (P)$ the BRW is recurrent and the limit set equals the full boundary) depends on the decay of the return probabilities. We make the following conjecture in believing that there is a more explicit connection between the Hausdorff dimension and the decay of the return probabilities, compare with \cite{hueter:00} and \cite{La:06}.

\begin{conj}
Consider BRW on a non-amenable Cayley graph. Then, the Hausdorff dimension of the limit set is continuous for $m\neq 1/\rho(P)$ and discontinuous at $m=1/\rho(P)$.
\end{conj}

\subsection{General graphs}
One natural direction to generalize our results is to consider general graphs (with bounded degrees). Since our results  depend on the homogeneity of the graph, there are basic questions that were not yet treated or are still unsolved.

One first question to ask about the trace of a BRW is whether the process eventually visits the whole graph almost surely. This is equivalent to the question of strong recurrence of the process and until now no \emph{nice} criterion for strong recurrence in general is known. In \cite{menshikov:97} they give a rather implicit criterion in terms of  Lyapunov functions. Another attempt in order to understand strong recurrence of BRWs is made in \cite{mueller08} where more references and details about this problem can be found. We state the conjecture made in \cite{mueller08}: let $\tilde\rho(P)=\inf \rho(P_F)$, where  the $\inf$ is over all induced connected subgraphs $F\subset G$ with finite boundaries.

\begin{conj}\label{conj:strrec}
Let $G$ be a graph with bounded degrees. Then, the BRW is strongly recurrent iff the mean number of offspring $m$ is larger than $1/\tilde\rho(P)$.
\end{conj}

Connected to the question of recurrence is the question if the trace is always a proper subgraph of the base graph. If the BRW is strongly recurrent, then  $\P(\Tr=G)=1$, and if it is weakly recurrent, then $0<\P(\Tr=G)<1$.
\begin{que}
Does the event $\{\Tr=G\}$ coincide with the event that the BRW returns infinitely many times to the origin?
\end{que}

In view of Lemma \ref{lem:weakrec} and the fact that the SRW on the trace of a transient BRW on a unimodular random graph has spectral radius $1$, one might ask the following:

\begin{que}
Is the spectral radius of the SRW on the trace of a BRW equal to $1$ a.s.? Is the trace a.s. an amenable graph?
\end{que}

Furthermore, we believe Theorem \ref{thm:SRWonBRWtrace} to hold in general:

\begin{conj}\label{conj:SRWonBRWtrace}
The trace of a transient BRW is a.s.~transient for SRW.
\end{conj}

Eventually, we state the conjecture made in \cite{BG-GL:07} in the following stronger form.

\begin{conj}\label{conj:tr_strrec}
Every BRW on a.e.~trace of a transient BRW is  strongly recurrent.
\end{conj}

\subsection{General family trees}
Another way of generalization is to consider more general family trees. Theorem \ref{thm:tr=unimod} naturally generalizes to traces of tree-indexed random walks on Cayley graphs where the family tree is a unimodular random  tree. For example we could use the trace of a BRW on a homogeneous tree as the family tree for another BRW or even iterate this procedure. In consideration of the rich behaviour of tree-indexed random walks in general, see \cite{benjamini:94}, it is interesting to study to which extent the results presented here hold in a more general setting of family trees.

\subsection{Random unimodular graphs}
It was recently proven in \cite{Elek:08} that any unimodular measure on the space of rooted trees can be obtained as an appropriate weak limit. While this question is open for unimodular measures on rooted graphs, one might be able to construct, e.g. using the mapping of the family tree to the base graph, a sequence of rooted finite graphs that converge to the trace of BRW.

\begin{que}
Can the trace of a transient BRW on a URG be obtained as the weak limit of finite rooted graphs?
\end{que}

Until now, no \emph{nice} examples for unimodular measures on rooted trees except for \UGW~ have been known, see \cite{KS:09}. Theorem \ref{thm:tr=unimod} applied to BRWs on  homogeneous trees deliver other examples for unimodular random trees. Another candidate is given by the construction in the proof of Theorem \ref{thm:SRWonBRWtrace}, as well as the trace of a bi-infinite SRW on a unimodular random tree. However, the latter has no \emph{interesting} boundary properties.

\begin{rem} In the way trace of BRW  was studied here one can consider the same questions for other related processes that exhibit phase transitions on non-amenable graphs, see \cite{La:06},~\cite{Lyons:00}.
\end{rem}

\subsubsection*{Acknowledgment}
The authors wish to thank Nicolas Curien and Nina Gantert for helpful comments on a first version of this note. We are also grateful to the referees for their careful reading and useful comments.

\begin{small}
\addcontentsline{toc}{chapter}{Bibliography}
\bibliography{bib}
\end{small}

\bigskip
 \noindent
\begin{tabular}{l}
Itai Benjamini\\
Faculty of Mathematics\\
Weizmann Institute of Science\\
Rehovot 76100\\ Israel\\
{\tt itai.benjamini@weizmann.ac.il}\\
\\ \\
Sebastian M\"uller \\
Laboratoire d'Analyse, Topologie, Probabilit\'es\\
Universit\'e de Provence\\
 39, rue F. Joliot Curie\\
13453 Marseille Cedex 13\\
France\\
{\tt mueller@cmi.univ-mrs.fr}\\
\end{tabular}

\end{document}